\documentclass{amsart}
\usepackage[T1]{fontenc}   

\newtheorem{theorem}{Theorem}[section]
\newtheorem{lemma}[theorem]{Lemma}
\newtheorem{proposition}[theorem]{Proposition}
\newtheorem{corollary}[theorem]{Corollary}

\theoremstyle{definition}
\newtheorem{definition}[theorem]{Definition}

\newtheorem{example}[theorem]{Example}
\numberwithin{equation}{section}

\usepackage{amsmath}
\usepackage{amssymb}
\usepackage{amsthm}
\usepackage{amsmath}

\newcommand{\R}{\mathbb{R}}
\newcommand{\Q}{\mathbb{Q}}
\newcommand{\g}{\mathfrak{g}}
\newcommand{\h}{\mathfrak{h}}
\newcommand{\Z}{\mathbb{Z}}
\newcommand{\T}{\mathbb{T}}
\newcommand{\N}{{\mathbb N}}
\newcommand{\X}{\operatorname{\mathcal X}}

\newcommand{\Aff}{\operatorname{Aff}}

\newcommand{\Aut}{\text{Aut}}
\newcommand{\bd}{\mathop{Bd}}

\usepackage{tikz}
\usetikzlibrary{arrows}
\usetikzlibrary{matrix}

\usepackage{pgfplots}

\usepackage{hyperref}

\begin{document}

\noindent                                             
\begin{picture}(150,36)                               
\put(5,20){\tiny{Submitted to}}                       
\put(5,7){\textbf{Topology Proceedings}}              
\put(0,0){\framebox(140,34){}}                        
\put(2,2){\framebox(136,30){}}                        
\end{picture}                                        
\vspace{0.5in}

\renewcommand{\bf}{\bfseries}
\renewcommand{\sc}{\scshape}
\vspace{0.5in}

\title[Semigroup Problem]%
{The Semigroup Problem for Central
Semidirect Products of $\R^n$ with $\R^m$}


\author{Kevin Lui}
\address{Department of Mathematics, University of California, Santa Barbara, 552 University Rd, Santa Barbara, CA 93106}
\email{kevinywlui@gmail.com}

\author{Viorel Nitica}
\address{Department of Mathematics, West Chester University, West Chester, PA 19380}
\email{vnitica@wcupa.edu}
\thanks{VN is the corresponding author. VN was partially supported by Simons Foundation Grant 208729.}

\author{Siddharth Venkatesh}
\address{Department of Mathematics, University of California Berkeley, Berkeley, CA 94720}
\email{sidnv91@gmail.com}


\thanks {This paper was written during the Summer 2013 program REU at Pennsylvania State University, supported by the NSF grant  DMS-0943603.  K. Lui and S. Venkatesh are undergraduate students. V. Nitica was one of the faculty coordinators. The authors would like to thank the anonymous referee for many corrections and suggestions that helped improve the quality of the paper. Several contributions of the referee, in particular Lemma 2.7, are mentioned in the text.}

\begin{abstract}
    In this paper we prove two new results about closed subsemigroups in the
    family of solvable Lie groups that are central semidirect products of $\R^m$ and $\R^n$, $H_{mn} := \R^{m} \ltimes_{\phi} \R^{n}.$
    An example of such group is the group of orientation preserving affine transformations of the line $\Aff^+$. 
    We assume that $\phi$, the structure homomorphism, is continuous and maps nontrivially into the center
    of $\Aut(\R^{n})$. The first result states that
    the closure of a subsemigroup generated by a subset in $H_{mn},$ that is not included in a maximal subsemigroup
    with nonempty interior, is actually a subgroup. The second result states that among the subsets in $H_{mn}$ that are
    not included in a maximal proper subsemigroup, those that generate $H_{mn}$ as a closed subsemigroup are dense.
    Results of this nature were obtained before only for abelian and nilpotent Lie groups and their compact extensions.
    As an application of the technique developed in the paper, we also find the minimal number of generators as a closed group and as a closed semigroup of $H_{mn}$.
\end{abstract}

\maketitle

\vspace{0.5cm}

\section{\bf Introduction}

We start with some generalities that put the problem studied in this paper in a larger context.

A typical problem in topological dynamics is to investigate (topologically) transitive transformations, that is,
continuous maps that have a dense forward orbit. One way to construct such transformations is to consider a map
that is already known to be transitive and to look at its extensions. Given a continuous transformation $f:\X\to \X$ of a topological space $\X$, a Lie group $G$, and
a continuous map $\beta:\X\to G$ called a {\it cocycle}, one can
define an {\em extension with fiber $G$ and base $(f,\X)$} by $f_{\beta}:\X\times G\to \X\times G, f_{\beta}(x,\gamma)=(f
x,\gamma\beta(x))$. Examples of transitive maps $f$ for which one can investigate the existence of transitive extensions are
provided by the restrictions $f:\X\to \X$
of a diffeomorphism $f$ of a compact manifold to a hyperbolic basic set $\X$. We refer to~\cite{KH}
for basic definitions and for an introduction to the theory of hyperbolic dynamic systems.

One can find in \cite{MNT1} a general conjecture about
transitivity of such extensions: modulo the obstruction that
the range of the cocycle is included in a maximal subsemigroup with nonempty
interior, the set of $C^r$ transitive cocycles contains an open and dense
subset of $C^r$ cocycles. The conjecture is proved for various classes of Lie groups $G$ that are
semidirect products of compact and abelian/nilpotent groups in
\cite{MNT1,MNT2,MNT3,MNT4,N-nil,NP,N-T-nil}.

The main goal of this paper is to focus on a related semigroup problem. The proof of transitivity of
$f_{\beta}$ in \cite{MNT1} is based on showing that the set $\mathcal{L}_{\beta}(x)$ of heights of $\beta$ over
a periodic point $x$ of $f$ is the whole fiber $G$. To obtain that $\mathcal{L}_{\beta}(x) =G$, one has to prove that for a typical family $F\in G^p$
that generates $G$ as a closed group and which is not contained in a maximal subsemigroup with
nonempty interior, $F$ generates $G$ as a closed semigroup as well. We refer to this
question as the \emph{Semigroup Problem}. The problem was solved for $G = \R^n$ \cite[Lemma 5]{NP} and more generally $G=K\times \R^n$ where $K$ is a compact Lie group \cite[Theorem 5.10]{MNT1}. It is
also solved for $G = SE(n)$ \cite[Theorem 6.8]{MNT1} and for Heisenberg groups in \cite[Theorem 8.6]{N-T-nil}.

In this paper, we solve the Semigroup Problem for a family of solvable groups related to the group of orientation preserving affine transformations of the line $\Aff^+$. We recall that $\Aff^+$ is the simplest non-nilpotent solvable Lie group and the unique simply connected, nonabelian, 2-dimensional Lie group. The groups under consideration are semidirect products $\R^{m} \ltimes_{\phi} \R^{n}$, where the structure homomorphism $\phi$ is continuous and maps into the center of $\Aut(\R^{n})$, the set of positive scalar matrices. To simplify the notation, we denote these groups by $H_{mn}$.

The structure of the maximal subsemigroups with nonempty interior in a simply connected solvable Lie group is described in a fundamental paper by Lawson~\cite{L}.
They are in one-to-one correspondence, via the exponential map, with closed half-spaces with boundary a Lie subalgebra~\cite[Theorem 12.5]{L}.

We now list some facts about the groups $H_{mn}$ which will be proved in Section 2.

1. Let $G_{n}$ denote the matrix group
$\left\{
    \begin{pmatrix}
        a & \mathbf{b} \\
        0 & I_n
    \end{pmatrix}
: \mathbf{b} \in \R^n, a> 0 \right\}.$
Note that $G_{1}$ is isomorphic as a Lie group to $\Aff^+$. We prove in Theorem \ref{Iso} that $H_{mn}$ is isomorphic to $\R^{m-1} \times G_{n}.$ Hence, after identifying the $a$ coordinate in $G_{n}$ with its natural logarithm, multiplication in $H_{mn}$ becomes:
\begin{equation}\label{eq:1}
(v, a, \mathbf{b})(v', a', \mathbf{b'}) = (v + v', a + a', \mathbf{b} + e^{a}\mathbf{b'})
\end{equation}
for $v,v' \in \R^{m-1}, a,a' \in \R, \mathbf{b}, \mathbf{b'} \in \R^{n}.$

2. Let $\h_{mn}$ be the Lie algebra of $H_{mn}$. As shown in Lemma \ref{exp}, the exponential map $\exp$ of $H_{mn}$ is given by
  \[
        \exp(v, a, \mathbf{b}) =
        \begin{cases}
            (v, a, \frac{\mathbf{b}}{a}(e^{a} - 1) )& \text{ if $a \not= 0$},\\
            (v, 0, \mathbf{b}) & \text{ if $a = 0$} .
        \end{cases}
    \]
where $(v, a, \mathbf{b}) \in \h_{mn}$. In particular $\exp$ is a bijection with its inverse analytic everywhere. Thus, $H_{mn}$ is exponential.

We solve the Semigroup Problem for $H_{mn}$ by proving two closely related conjectures. We first introduce some terminology and recall the fundamental result of Lawson.

\begin{definition}
    \label{max}
    A \emph{maximal subsemigroup with nonempty interior} of a topological group
    $G$ is a proper subsemigroup $M$ of $G$ with nonempty interior such that
    $M$ is not a group and the only subsemigroups of $G$ containing $M$ are $G$
    and $M$. In this paper, the term \emph{maximal subsemigroup} will always refer to
    those with nonempty interior.
\end{definition}

We note that maximal subsemigroups in a connected topological group are closed~\cite[Proposition 5.4]{L}.

\begin{definition}
    \label{sep}
    A subset $S$ of a topological group $G$ is called \emph{nonseparated} if it
    is not contained in a maximal subsemigroup.
\end{definition}

\begin{definition}
    \label{good}
    A subset $S$ of a Lie group $G$ is called \emph{multiplicatively quasi-dense} if the closure of the
    semigroup it generates is a group not contained in any connected
    codimension 1 subgroup.
\end{definition}

\begin{definition}
   \label{great}
   A subset $S$ of a topological group $G$ is called \emph{multiplicatively dense} if it generates a dense
   subsemigroup. Additionally, for a fixed positive integer $\ell$, we define $p \in G^{\ell}$ to be a
   \emph{multiplicatively dense $\ell$-tuple} if the subset corresponding to $p$ is multiplicatively dense.
\end{definition}

The following result belongs to Lawson \cite[Theorem 12.5]{L}.

\begin{theorem}\label{thm-lawson01} The maximal subsemigroups $M$ with non-empty interior of a simply connected Lie group $G$ with Lie algebra $L(G)$ and with $G/\text{Rad}\;G$ compact are in one-to-one correspondence with their tangent objects
\[
L(M)=\{x\in L(G):\exp tx\in M, t\ge 0\},
\]
and the latter are precisely the closed half-spaces with boundary a subalgebra. Furthermore, $M$ is the semigroup generated by $\exp(L(M))$.
\end{theorem}

\begin{definition}
    \label{border subgroups}
    For groups $G$ that satisfy Theorem~\ref{thm-lawson01}, we call the boundaries
    of their maximal subsemigroups \emph{border
    subgroups}.
\end{definition}

We now state two general semigroup conjectures for Lie groups, that we will solve for $G=H_{mn}$.

\textbf{Semigroup Conjecture 1:} Let $G$ be a Lie group. Then, $S
\subseteq G$ is multiplicatively quasi-dense if and only if it is not contained in a maximal subsemigroup.

\textbf{Semigroup Conjecture 2:} Let $G$ be a Lie group. Then, for each
positive integer $\ell$, the set of multiplicatively dense $\ell$-tuples in $G^{\ell}$ is dense in the set of
nonseparated $\ell$-tuples in $G^{\ell}$.

The first conjecture says that if $S \subseteq H_{mn}$ is not contained in a maximal subsemigroup, then the closed subsemigroup it generates is the same as the closed subgroup it generates. The second conjecture says that for a typical such subset, $H_{mn}$ is the closed subgroup generated by the subset. Together they imply a positive solution for the Semigroup Problem for $H_{mn}$.

The rest of the paper is organized as follows. In Section~\ref{s:main} we present the main results. We first prove the general facts about the structure of $H_{mn}$ listed in the introduction and use them to prove some parts of the Semigroup Conjecture 1. Then we describe certain Lie group automorphisms of $H_{mn}$. Next, we prove an exact sequence lemma which is used to prove the conjectures inductively. Then we prove some properties of nonseparated subsemigroups and some properties of products in $H_{mn}$ and use them to finish the proof of Semigroup Conjecture 1. Finally, we use a theorem of Kronecker and a structure lemma about multiplicatively quasi-dense subsets in $H_{mn}$ to prove Semigroup Conjecture 2. In Section~\ref{s:examples} we show several examples of nonseparated subsemigroups that are not groups, further validating our main results.

A spinoff of our work is a procedure to construct examples of dense subsemigroups in a large class of closed subgroups of $H_{mn}$ or $G_n$. One starts with a nonseparated subset $F$ in $H_{mn}$ that generates a closed subgroup $G\subseteq H_{mn}$. If the subsemigroup $S$ generated by $F$ is not a group (this hypothesis is easy to check for the simple examples presented in Section 3), then $S$ is a dense subsemigroup in $G$. In particular, one can consider the case $G=H_{mn}$ or $G=G_n$ and construct examples of dense semigroups.

To complete the picture, in Section~\ref{s:generators} we deduce the minimal numbers of generators as a closed semigroup ($n+2$ and $\max\{m,n+2\}$) and as a closed subgroup ($n+1$ and $\max\{m,n+1\}$) for $G_n$ and $H_{mn}$ respectively. Finding these numbers for various Lie groups is a question of current research interest that was mostly overlooked in the literature. See~\cite{J},~\cite{AV} for related results in this direction.

\section{\bf Main results}\label{s:main}
We start this section by proving the structure theorem stated in the introduction.

\begin{theorem}
    \label{Iso}
    The group $H_{mn}=\R^{m} \ltimes_{\phi} \R^{n}$ is isomorphic as a Lie group to $\R^{m -1}\times G_{n}.$
\end{theorem}
\begin{proof}
Note that we assumed that $\phi$ maps continuously into the group of positive scalar matrices, which is isomorphic to $\R$. Any continuous additive homomorphism $\phi$ from $\R^{m}$ to $\R$ is  $\Q$-linear and hence $\R$-linear. Thus, $\phi$ is a linear map from $\R^{m}$ to $\R$ and hence $\R^{m} \cong \ker(\phi) \oplus H$, as a Lie group (with $H \cong \R$ a subspace in $\R^{m}$). We denote $\ker(\phi)$ by $K$. Now, let $h$ be a nonzero element in $H$, so that $H = \R h$, and let $r.I_{n} = \phi(h)$, with $r \in \R^{+}$. We then define a map
$$\Phi: (K \oplus H) \ltimes \R^{n} \rightarrow \R^{m-1} \times G_{n}$$
which sends $(v, xh, \mathbf{b}) \mapsto (v, x \ln r, \mathbf{b}).$ This map is an isomorphism of smooth manifolds, hence we simply need to verify that it is a group homomorphism. This follows from the following computation:
\begin{align*}
\Phi((v, xh, \mathbf{b})(v', x'h, \mathbf{b'})) &= \Phi(v + v', (x + x')h, \mathbf{b} + e^{x}r\mathbf{b'})\\
&= (v + v', (x + x')\ln r, \mathbf{b} + e^{x}r \mathbf{b'})\\
&= (v, x\ln r, \mathbf{b})(v', x'\ln r, \mathbf{b'})\\
&= \Phi(v, xh, \mathbf{b})\Phi(v', x'h, \mathbf{b'}).
\end{align*}
\end{proof}

From now on we assume $H_{mn} = \R^{m-1} \times G_{n}$ with the product given by~\eqref{eq:1}.
It follows from Theorem~\ref{Iso} that the groups
$H_{mn}$ are solvable. Additionally, as $H_{mn}$ is diffeomorphic to $\R^{m+n}$, it has to be
simply connected. Hence, $H_{mn}$ satisfies the hypothesis of \cite[Theorem
12.5]{L}.

\begin{lemma}
    \label{exp}
    The exponential map $\exp:\h_{mn} \rightarrow H_{mn}$ is an analytic bijection
    with analytic inverse.
\end{lemma}
\begin{proof}
    Note that $\exp$ is always analytic. We show that it is bijective. Since
    $H_{mn} \cong \R^{m-1} \times G_{n}$, it will suffice to show that the restriction of $\exp$ to $G_{n}$
    is bijective with analytic inverse. Suppose $\g_{n}$ is the Lie algebra of
    $G_{n}$.
    Then, $\g_{n}$ consists of elements $(a,\mathbf{b})$, $a\in \R, \mathbf{b} \in \R^{n}$, with multiplication $(a,
    \mathbf{b})(a', \mathbf{b'}) = (aa', a\mathbf{b'}).$ Thus $(a, \mathbf{b})^\ell = (a^{\ell}, a^{\ell-1}\mathbf{b}).$
    A calculation shows:
    \[
        \exp(a, \mathbf{b}) =
        \begin{cases}
            (a, (e^{a} - 1)\frac{\mathbf{b}}{a} )& \text{ if $a \not= 0$},\\
            (0, \mathbf{b}) & \text{ if $a = 0$},
        \end{cases}
    \]
    and $\exp$ is clearly bijective.  Consider now $\log = \exp^{-1}$,
    \[
        \log( x, \mathbf{y}) =
        \begin{cases}
            (x, \frac{x}{e^{x} - 1}\mathbf{y}) & \text{ if $x \not= 0$},\\
            (0, \mathbf{y}) & \text{ if $x = 0$}.
        \end{cases}
    \]
    which is clearly analytic for $x \neq 0$. For $x = 0$, the Taylor series
    of $x$ and $e^{x} - 1$ have zero constant terms and non-zero $x$ terms. Thus,
    the denominator in formula above is a zero of order 1 and by standard complex analysis results
    it follows that the fraction is analytic. We conclude that $\log$ is analytic
    everywhere.
\end{proof}

Lemma~\ref{exp} and~\cite[Theorem 12.5]{L} imply that the border groups in $H_{mn}$ are
connected codimension 1 subgroups, images of hyperplane subalgebras of $\h_{mn}$.
Let $H$ be a connected codimension 1 subgroup of $H_{mn}$, and  $\h$ its is a
hyperplane Lie subalgebra. If $S\subseteq H$, then $S$ is contained
in the maximal semigroups given by exponentiating each of the half-spaces
determined by $\h$. This gives the following theorem.

\begin{theorem}
    \label{s1rev1}
    Let $S \subseteq H_{mn}$ be nonseparated. Then, $S$ is not contained in any
    connected codimension 1 subgroup of $H_{mn}$.
\end{theorem}

We show the forward implication of Semigroup Conjecture 1.

\begin{theorem}
    \label{s1fwd}
    Let $S \subseteq H_{mn}$ be multiplicatively quasi-dense. Then, $S$ is nonseparated.
\end{theorem}

\begin{proof} If $M\subseteq H_{mn}$ is a maximal subsemigroup and $S$ is multiplicatively quasi-dense, then $S$ is not contained in $\bd(M)$,
    the boundary of $M$. Now, suppose for contradiction that $S$ is multiplicatively quasi-dense but $S$ is separated. Then, $S$ is contained in some maximal
    subsemigroup $M$ of $H_{mn}$ but not in $\text{Bd}(M)$. Hence, there is some $x \in S$ which is in the interior of $M$. Now, if $U$ is the subsemigroup generated by $S$, then as $S$ is multiplicatively quasi-dense, $x^{-1} \in \overline{U}$. But $\log(x), \log(x^{-1})$ are negatives of each other and are hence on opposite sides of $\log(\text{Bd}(M)).$ Hence, $\overline{U}$ is not contained in $M$, a contradiction of the fact that $M$ is closed. Thus, $S$ is nonseparated.
\end{proof}

We describe the (Lie group) automorphisms of $H_{mn}=\R^{m-1} \times G_{n}$.
They will play an important
role in the proof of the conjectures as the properties of a subset being
nonseparated, multiplicatively quasi-dense or multiplicatively dense are preserved by automorphisms.
As $H_{mn}$ is simply
connected, one can consider the Lie algebra automorphisms.
Let $\{X_1,\dots,X_{m-1}\}$ be the canonical basis in $\R^{m-1}$, $Y$ the canonical basis in $\R$ and
$\{Z_1,\dots,Z_{n}\}$ the canonical basis in $\R^{n}$.
Then a basis for $\h_{mn}$ is given by the set $\{X_{1}, \ldots, X_{m-1}, Y, Z_{1}, \ldots, Z_{n}\}$. The commutation relations are $$[X_{i}, X_{j}] = [X_{i}, Y] =
[X_{i}, Z_{j}] = [Z_{i}, Z_{j}] = 0, [Y, Z_{i}] = Z_{i}.$$ Thus, we immediately
get the following automorphisms:
\begin{itemize}
\item Type A: Any automorphism $\psi$ of $\R^{m-1}$ extends to an
automorphism that is the identity on $G_{n}$.
\item Type B: Any automorphism $\psi$ of $\R^{n}$ extends to an automorphism
which is the identity on the $X_i$ and $Y$ coordinates.
\item Type C: The map that fixes $X_{i}, Z_{j}$ and sends $Y \mapsto Y +
\sum_{i}\alpha_{i}X_{i} + \sum_{j} \beta_{j}Z_{j}$ is a Lie algebra
automorphism. Hence, for a given $(v, a, \mathbf{b}) \in H_{mn}$ there is a Lie
group automorphism that sends it to $(0, a, \mathbf{0})$.
\end{itemize}

We are ready to begin proving the other implication in the Semigroup Conjecture 1 for $H_{mn}$. Our
approach is by induction. We rely heavily on the following Lemma, which
is proved in the setting of a general topological group. If $G$ is a group and $X\subseteq G$, we denote by $U(X)$ the semigroup generated by $X$.
When the subset $X$ is understood from the context, we frequently abbreviate $U(X)$ by $U$.
\begin{lemma}
    \label{exact}
   Consider an exact sequence of topological groups:
    \[
        \begin{tikzpicture}
\matrix (m) [matrix of math nodes,row sep=1em,column sep=1.5em,minimum width=2em]
{ 0 & B & G & A & 0.\\};
\path[-stealth]
(m-1-1) edge (m-1-2)
(m-1-2) edge(m-1-3)
(m-1-3) edge node[auto] {$\pi$}(m-1-4)
(m-1-4) edge (m-1-5);
\end{tikzpicture}
\]
Let $S \subseteq G$ and let $U$ be the semigroup generated by $S$. Then:
\begin{enumerate}
    \item
        $\overline{\pi(\overline{U})} = \overline{\pi(U)}$.
    \item
        If $\overline{U}$ is a group, then so is $\overline{U(\pi(S))}$.
    \item
        If $S$ is multiplicatively dense in $G$, then $\pi(S)$ is multiplicatively dense in $A$.
    \item
        If ${\pi(\overline U)}$ is a group and
        $\overline{U} \cap B$ is a group, then $\overline{U}$ is a group.
    \item
        If $\pi(S)$ is multiplicatively dense in $G$, $\pi(\overline{U})$ is closed in $A$ and
        $\overline{U} \cap B$ is multiplicatively dense in $B$, then $S$ is multiplicatively dense.
    \item
        If $S$ is nonseparated, then so is $\pi(S)$.
    \item
        Let $B = \R^{n}$ and $G = \R^{n} \rtimes A$, where $A$ second countable group, and
        suppose $\overline{U} \cap B$ is multiplicatively quasi-dense in $A$. Then $\pi(\overline{U})$ is
        closed in $A$.
\end{enumerate}
\end{lemma}
\begin{proof}
    ~\begin{enumerate}
        \item
            Note that $\pi(U) = U(\pi(S))$. Clearly
            $\overline{\pi(\overline{U})} \supseteq \overline{\pi(U)}$. For the
            reverse, pick $x \in \overline{U} \backslash U$. Let $y = \pi(x)$.
            Let $V$ be a neighborhood of $y$. Then, $\pi^{-1}(V)$ is a
            neighborhood of $x$, which hence contains an element of $U$. Thus,
            $V$ contains an element of $\pi(U)$. Hence, $\pi(\overline{U}) \subseteq
            \overline{\pi(U)}$ and thus $\overline{(\pi(\overline{U}))} = \overline{\pi(U)}.$

        \item
            Note that $\overline{U(\pi(S))} = \overline{\pi(U)} =
            \overline{\pi(\overline{U})}$, which is a group if $\overline{U}$
            is a group.

        \item
            If $\overline{U} = G$, then
            $\overline{U(\pi(S))} = \overline{\pi(U)}=A$, as by (1) it contains $\pi(\overline{U}).$

        \item
            In order for $\overline{U}$ to be a group, every element in $\overline{U}$
            must have an inverse in $\overline{U}$. Pick $x \in \overline{U}$. Then, there
            exists an inverse in $\pi(\overline{U})$ for $\pi(x)$. Thus, there
            exists $y \in \overline{U}$ such that $yx \in \overline{U} \cap B$.
            Thus, $x^{-1}y^{-1} \in \overline{U} \cap B$. Thus, $x^{-1} =
            x^{-1}y^{-1}y \in \overline{U}$.

        \item
            Since $\pi(\overline{U})$ is closed, and $\pi(S)$ is multiplicatively dense,
            $\pi(\overline{U})=A$. Additionally, as $\overline{U} \cap
            B$ is multiplicatively dense in $B$, it must be all of $B$, as it is a closed semigroup.
            As $\overline{U}$ contains $B$ and a representative for each
            coset in $A$, $\overline{U} = G$.

        \item
            Let $M$ be a maximal semigroup of $A$.
            Then, by \cite[Lemma 3.12]{L}, $\pi^{-1}(M)$ is a maximal semigroup
            of $G$ and it has nonempty interior by the continuity of
            $\pi$. As $S$ is nonseparated, $S$ is not contained in
            $\pi^{-1}(M)$. Thus, $\pi(S)$ is not contained in $M$. Hence,
            $\pi(S)$ is nonseparated.

        \item
           Let $y$ be a limit point of $\pi(\overline{U})$. By the second countability of $A$ and hence of $G$, we may choose a sequence $\{(x_{i}, y_{i})\} \subseteq \overline{U}$ such that $y_{i} \rightarrow y.$ As $\overline{U} \cap B$ is multiplicatively quasi-dense in $B$, it is not contained in a hyperplane in $B$ and it is a group. Thus, it contains $n$ linearly independent vectors $\{z_{1}, \ldots, z_{n}\}$ in $B$ together with their negatives. As left multiplication by elements of $B$ does not change $y_i$, by multiplying each of the $(x_{i}, y_{i})$ on the left by suitable multiples of $z_{j}$ or $-z_{j}$, we can construct a new sequence $\{(x'_{i}, y_{i})\} \subseteq \overline{U}$ such that $\{x'_{i}\}$ is bounded. Choosing a convergent subsequence of $\{x'_{i}\}$ and taking the limit gives $(x, y) \in \overline{U}$. Hence, $\pi(\overline{U})$ is closed.
    \end{enumerate}
\end{proof}

Now we are going to analyze the border subgroups in $G_{1}$.

\begin{lemma}
    \label{G_1 borders}
    Identifying $G_{1}$ with its Lie algebra $\g_{1}=\R^2$, the border
    subgroups are the (boundary) curves $y = l(e^{x} - 1), (x,y)\in \g_1,$ where $l \in \R$,
    and $x = 0$. We call $l$ the slope of the boundary curve. Moreover,
    \begin{enumerate}
        \item
            Every nonzero point in $\R^{2}$ belongs to a unique boundary curve.
        \item
            For $l\ge 0$, the boundary curve is contained in the first and
            third quadrants. For $l\le 0$, the boundary curve is contained
            in the second and fourth quadrants.
        \item
            Let $l \le 0$. For points $z$ in the fourth quadrant, if
            $l_{z}$ is the slope of the boundary curve through $z$, then $l_{z}
            < l$ if and only if $z$ is below the boundary curve of slope $l$. For
            $z$ in the second quadrant, $z$ is below the boundary curve
            of slope $l$ if and only if $l_{z} > l.$
    \end{enumerate}
\end{lemma}
\begin{proof}
    A codimension 1 subspace of $\g_{1}$ is $1$-dimensional, so it is a subalgebra.
    The border subgroups are images under $\exp$ of these
    subalgebras. If the subalgebra is $\{t(a, b): t \in \R\}$, then by
    Lemma~\ref{exp}, the associated boundary curve is given by $\{(ta,
    \frac{b}{a}(e^{ta} - 1)) : t \in \R\}$ which is the curve $y = l(e^{x} -
    1), l\in \R,$  or $x = 0$. Now, (1) and (2) follows from the definition of the curves. For (3), note that
    for $z = (x_{0}, y_{0})$, $l_{z} = \frac{y}{e^{x_{0}} - 1}$, and hence for
    $z$ in the fourth quadrant (resp. second quadrant), $z$ is below the curve
    $y = l(e^{x} - 1)$ if and only if $l > l_{z}$ ($l < l_{z}$) as
    $e^{x_{0}} - 1 > 0$ ($< 0$).
\end{proof}

\begin{center}
        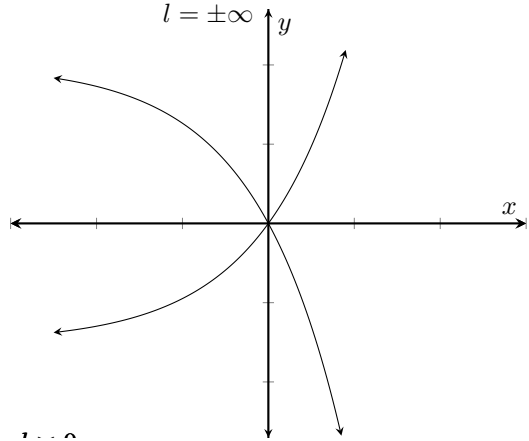
\begin{figure}[h]
        \begin{tikzpicture}[>=stealth]
            \begin{axis}[
                    xmin=-3,xmax=3,
                    ymin=-2.7,ymax=2.7,
                    axis x line=middle,
                    axis y line=middle,
                    axis line style=<->,
                    xlabel={$x$},
                    ylabel={$y$},
                    xticklabels={,,},
                    yticklabels={,,},
                    samples=50
                ]
                \addplot[no marks,<->]
                expression[domain=-2.5:0.85,samples=100]{-2*(exp(x)-1)}
                    node[pos=0,anchor=south west]{$l<0$};
                \addplot[no marks,<->]
                expression[domain=-2.5:0.9,samples=100]{1.5*(exp(x)-1)}
                    node[pos=-0.1,anchor=south west]{$l>0$};
                \addplot[no marks,thick,<->]
                expression[domain=-3:3,samples=10]{0}
                    node[pos=0,anchor=south west]{$l=0$};
                \addplot[no marks,thick,<->] coordinates {
                    (0,3)
                    (0,-3)
                };
                \addplot[no marks,nodes near coords={$l=\pm\infty$}]
                coordinates {
                    (-.7,2.4)
                };
            \end{axis}
        \end{tikzpicture}
        \caption{Different types of boundary curves}
    \end{figure}
    \end{center}

\begin{lemma}
    \label{H_mn products} Let $\pi$ be projection of $H_{mn} \cong \R^{m} \ltimes \R^{n}$
    onto $\R^{m}$, $S \subseteq H_{mn}$ with $\pi(S)$ nonseparated, $U=U(S)$ and $z_{0} = (w, a, \mathbf{b}),$ $z =
    (w', a', \mathbf{b'})\in \overline{U}$ with $a < 0, a' > 0$. Then
    $\left(0, 0, \frac{\mathbf{b}}{1 -e^{a}} + \frac{\mathbf{b'}}{e^{a'} - 1}\right)\in \overline{U}$.
\end{lemma}

\begin{proof}
    Let $v_{0} = (w, a), v = (w', a').$  As $\pi(S)$ is nonseparated,
    there exists a finite subset $F=\{v_{1}, \ldots, v_{l}\}$ of $\pi(S)$ such
    that the interior of its convex hull contains $0$. Denote $v = v_{l+1}, z = z_{l+1}$. Then for small
    $\alpha_{0}, \alpha_{l+1} > 0$, $-\alpha_{0} v_{0} - \alpha_{l+1}
    v_{l+1}$ is in the convex hull of $F$ and there
    exist $\alpha_{i} \ge 0, 1\le i\le l,$ such that
    $\displaystyle\sum_{i=0}^{l+1} \alpha_{i} v_{i} = 0$. Without loss of generality, assume
    $\alpha_{i} > 0$. Thus
    $\displaystyle\sum_{i= 0}^{l+1} t\alpha_{i} v_{i} = 0$ for $t$ any positive
    integer. Denote by $\hat{t}$ the projection of $(t\alpha_{0}, \ldots,
    t\alpha_{l+1})$ in the $(l+2)$-torus $\T^{l+2}=\R^{l+2}/\Z^{l+2}$.
    As $\T^{l+2}$ is compact group and the set $\{\hat{t}:t\in \N\}$ is a semigroup, there exists a subsequence $t_{k}$
    such that $\hat{t_{k}} \rightarrow 0.$ Denote by $t_{i}^{(k)}$ the
    integer closest to $t_{k}\alpha_{i}$. Then, as $t_{k} \alpha_{i}$ is
    increasing and unbounded, $t_{i}^{(k)} \rightarrow \infty$. Now, as $t_{k}
    \alpha_{i}$ converges to $0$ in $\T$, we note that for each $\epsilon > 0$,
    there exists an $k_{0} > 0$ such that for $k > k_{0}$, there exists
    some integer $r$ such that $|r - t_{k} \alpha_{i}| < \epsilon$. But then,
    $|t_{i}^{(k)} - t_{k} \alpha_{i}| < \epsilon,$ thus
    $t_{i}^{(k)} - t_{k}\alpha_{i} \rightarrow 0$ and $\displaystyle
    \sum_{i=0}^{l+1} t_{i}^{(k)} v_{i} \rightarrow 0.$

We now assume without loss of generality that $v_{1}, \ldots, v_{l}$ are in nondecreasing order with respect to their $a_{i}$ coordinate. Since $v_{i} \in \pi(S)$ there exists $z_{i} = (v_{i}, b_{i}) \in S$. By induction $z_{i}^{t} = (tv_{i}, b_i\frac{1 -
    e^{ta_{i}}}{1 - e^{a_{i}}})$ with the fraction equal to $t$ for $a_{i} =
    0.$ Thus,
    \[
        z_{0}^{t_{0}^{(k)}}\cdots z_{l+1}^{t_{l+1}^{(k)}} = \left(\sum_{i}
        t_{i}^{(k)} v_{i}, \sum_{i} \frac{1 - e^{t_{i}^{(k)}
        a_{i}}}{1 - e^{a_{i}}} e^{\sum_{j=0}^{i-1} t_{j}^{(k)}a_{j}}b_{i}\right).
    \]
    Now, for $i \not= l+1$, $\displaystyle\sum_{j=0}^{i} t_{j}^{(k)}a_{j}
    \rightarrow -\infty.$ This is clear when $a_{i} \le 0$ and holds for $a_{i}
    >0$ as well, since otherwise the sum for $i = l+1$ would converge to $\infty$ and not $0$ as
    constructed. Thus,
    for $i \not= 0, l+1$, $a_{i}\not= 0$,
    \[
        \frac{1 - e^{t_{i}^{(k)} a_{i}}}{1
        - e^{a_{i}}} e^{\sum_{j=0}^{i-1} t_{j}^{(k)}a_{j}} =
        \frac{e^{\sum_{j=0}^{i-1} t_{j}^{(k)}a_{j}} -
        e^{\sum_{j=0}^{i}t_{j}^{(k)}a_{j}}}{1 - e^{a_{i}}} \to 0
    \]
    and for $i \not= 0, l+1, a_{i} = 0$, as exponentials dominate over a linear
    term, we get the same result. Hence,
    \[
        z_{0}^{t_{0}^{(k)}}\cdots z_{l+1}^{t_{l+1}^{(k)}}
        \to \left(0, \frac{\mathbf{b_{0}}}{1 - e^{a_{0}}} +
        \frac{\mathbf{b_{l+1}}}{e^{a_{l + 1}} - 1}\right)
        = \left(0, \frac{\mathbf{b}}{1- e^{a}} + \frac{\mathbf{b'}}{e^{a'}-1}\right)\in \overline{U}.
\]
\end{proof}

For the following lemma we use an approach suggested by the referee, that simplifies our original proof.
It is convenient to identify $G_1$ with the open right hand half-plane endowed with the operation $(a,b)\cdot(c,d)=(ac, ad+b).$
Note that the topology on $G_1$ coincides with the induced topology from $\R^2$ and the operation $\cdot$ can be extended continuously to the whole plane.
The first quadrant in $G_1$ is identified with the set
of points in the upper half-plane to the right of $x=1$, the second quadrant with the set of points in the upper half-plane with
$x$-coordinate greater than $0$ and to the left of $x=1$, the third quadrant with the set of points in the lower half-plane with
$x$-coordinate greater than $0$ and to the left of $x=1$,  and the fourth quadrant with the set
of points in the lower half-plane to the right of $x=1$. The connected $1$-dimensional subgroups are precisely the
lines passing through $(1,0)$ intersected with $G_1$. If $(x,y)\in G_1$ with $0<x<1$, then $\{(x,y)^n\}$ converges to $(0, y/(1-x))$.

\begin{lemma}\label{lemma-quadrants01} Let $U$ be a nonseparated subsemigroup of $G_1$. For each open quadrant of $G_1$, the set of points belonging to $U$ in that quadrant has an unbounded set of $y$-coordinates. In particular, for each quadrant $Q$ in $G_1$, the set $U\cap Q$ is nonempty.
\end{lemma}

\begin{proof} Consider the set $B:=\{b:(0,b)\in \overline{U}\},$ where $\overline{U}$ is the closure of $U$ in $\R^2$. The set $B$ is non-empty. Indeed, the right half-plane $x\ge 1$ is a maximal subsemigroup, thus there exists $(x,y)\in U$ with $0<x<1$. As $\{(x,y)^n\}$ converges to $(0, y/(1-x))$, one has $y/(1-x)\in B$. If $B$ is not bounded above, then the lemma is true for the second quadrant. Indeed, one can chose a sequence $\{(0,b_n)\}_n$ with $b_n\in B, b_n\to \infty, \{b_n\}_n$ strictly increasing and with $\vert b_{n+1}-b_n\vert>C>0$, and then approximate each $(0,b_n)$ close enough by an element in $U$. Otherwise, assume that it has a least upper bound $b^*$. Then $(0,b^*)\in \overline{U}$, so one can find a sequence
$\{(x_n,y_n)\}\subseteq U$ converging to $(0,b^*)$. By our hypothesis on $U$, there exists $(c,d)\in U$ that is not contained in the maximal semigroup $H\cap G_1$, where $H$ is the closed half plane lying below the line determined by $(0,b^*)$ and $(1,0)$. Note that the point $(c,b^*-b^*c)$ lies on the line determined by $(0,b^*)$ and $(1,0),$ and hence since $(c,d)$ is above this line, $\varepsilon:=d-(b^*-b^*c)>0.$ We thus obtain:
\begin{equation*}
\begin{gathered}
(c,d)\cdot (x_n,y_n)\to (c,d)\cdot (0,b^*)=(0,b^*c+d)\\
=(0,d-(b^*-b^*c)+b^*)=(0,b^*+\varepsilon),
\end{gathered}
\end{equation*}
a contradiction to the fact that $b^*$ is the supremum of $B$.

Since $(x,y)\to (x,-y)$ is a topological isomorphism of $G_1$, hence preserves the property of nonseparated semigroup, one concludes that the lemma is true for the third quadrant. Since inversion in $G_1$ is a topological anti-isomorphism carrying the second quadrant to the fourth and the third to the first, we conclude that the lemma is true for all four open quadrants.
\end{proof}

\begin{lemma}
    \label{separation}
    Let $S \subseteq \R^{n}$. If for each $1$-dimensional subspace $X\subseteq \R^n$ and each projection $\tau: \R^{n} \rightarrow X$, $\tau(S)$ is nonseparated, then $S$ is nonseparated.
\end{lemma}
\begin{proof}
    Suppose, by contradiction, that $S$ is contained in a closed half-space
    $M$.  Let $H$ be the boundary hyperplane. Choose a complementary
    $1$-dimensional subspace $X$ and let $\tau$ be the projection onto $X$ with kernel
    $H$. Then, without loss of generality, $\tau(S) \subseteq [0, \infty)$ as
    $S$ is contained on one side of $H$, a contradiction.
\end{proof}

We now prove the following theorem, which together with Theorems \ref{s1rev1} and
\ref{s1fwd} proves the Semigroup Conjecture 1 for $H_{mn}$.

\begin{theorem}
    \label{s1rev2}
    Let $S \subseteq H_{mn}$ nonseparated, $U=U(S)$.  Then $\overline{U}$ is a group.
\end{theorem}
\begin{proof}
Note that $H_{mn} \cong \R^{m} \rtimes \R^{n}$ is the central term in a short exact sequence with image $\R^{m}$ and kernel $\R^{n}$. We denote the kernel by $K$ and we denote $\R^{m} = \R^{m-1} \times A$ where $A$ is the $a$ coordinate in $G_{n}$ under the isomorphism $H_{mn}\cong \R^{m-1} \times G_{n}$.

If we show that $\overline{U} \cap K$ is multiplicatively quasi-dense in $K$, then by Lemma 2.5 (1) and (7), $\pi(\overline{U}) = \overline{\pi(U)}$ is a group, and then by Lemma 2.5 (4), $\overline{U}$ is a group. Thus, it suffices to show that $\overline{U} \cap K$ is multiplicatively quasi-dense in $K$, which due to \cite[Lemma 2.12]{MNT1} is equivalent to showing that $\overline{U} \cap K$ is nonseparated in $K$.

After applying a type C automorphism we may assume that $\overline{U}$ contains $z_{0} = (0, a_{0}, 0)$ for some $a_{0} \in \R^{+}$. Denote $\overline{U} \cap K$ by $S'$. By Lemma 2.9, in order to show that $S'$ is separated, it suffices to show that for any one-dimensional subspace $X$ and any projection $\tau$ from $K$ to $X$, $\tau(S')$ is nonseparated.
So let $X = \text{span} \{B_{1}\}$ and let $\{B_{2}, \ldots, B_{n}\}$ be a basis for $\ker(\tau)$. Let $\pi$ be the projection from $H_{mn}$ onto the span of $A$ and $X$. Then, $\pi$ is a surjective Lie group homomorphism onto $G_{1}$. By Lemma 2.5 (6), its image is nonseparated.
Hence by Lemma 2.8, we find $(a, b), (a', b') \in \pi(U)$ such that $a, a', b' < 0$, $b > 0$. Thus, $U$ contains elements $z = (x,a,\mathbf{b}), z' = (x', a', \mathbf{b'})$ with $\mathbf{b}$ and $\mathbf{b'}$ having $B_{1}$ coordinate $b, b'$ respectively.
Thus, applying Lemma 2.7 to $z_{0}, z$ and to $z_{0}, z'$, we conclude that $S'$ contains $(0, 0, \frac{\mathbf{b}}{1 - e^{a}})$ and $(0, 0, \frac{\mathbf{b'}}{1 - e^{a'}}).$ Thus, $\tau(S')$ contains positive and negative elements and is hence nonseparated.
\end{proof}

\begin{corollary} A subset of $H_{mn}$ is multiplicatively quasi-dense if and only if it is non-separated.
\end{corollary}

We will now focus on proving the Semigroup Conjecture 2 for $H_{mn}$.

\begin{lemma} \label{good structure} Let $S \subseteq H_{mn}$ be multiplicatively quasi-dense. Then, there exists an automorphism $\Phi$ of $H_{mn}$, such that $(x, \ln a, \mathbf{0}), (x, \ln c{i}, |1 - c_{i}| \mathbf{e_{i}}) \in \Phi(S')$, where $a > 1, \mathbf{e_{i}}$ is the $i$th-vector in the standard basis for $B \cong \R^{n}$ and $S'\subseteq H_{mn}$ is an arbitrarily small perturbation of a finite set in $S$. Additionally, we can choose $\Phi$ such that either $(x', \ln a', 0) \in \Phi(S')$ for $a' < 1$ or $c_{1} < 1$.
\end{lemma}

\begin{proof} As Lie group automorphisms are continuous, we may perturb $S$ before or after applying an automorphism of $H_{mn}$. We get $(x, \ln a, 0) \in \phi(S)$ by a type C automorphism $\phi$. Next we prove the following statement for $S$ by induction on $i$: for each $i$ from $1$ to $n$, there exists an automorphism $\Phi_{i}$ such that $(x_j, \ln c_j, \mathbf{e_{j}}) \in \Phi_{i}(S)$ for all $j$ from $1$ to $i$.

Base Case i = 1: By the nonseparation of $\phi(S)$, there exists $(x_{1}, \ln c_{1}, \mathbf{v})$ in $\phi(S)$ with $\mathbf{v} \not = 0$. We send this by a type $B$ automorphism $\Phi_{1}$ to $(x_{1}, \ln x_{1}, \mathbf{e_{1}}).$ Additionally, if there exists any such element with $c_{1} < 1$, we can choose $c_{1} < 1$. Otherwise by nonseparation of $\phi(S)$, there exists $(x', \ln a', \mathbf{0}) \in \phi(S)$ for $a' < 1$ which is unchanged by $\Phi_{1}$ and any of the following $\Phi_{i}$ as these are type $B$ automorphisms.

Inductive Step: Suppose the statement holds for $i - 1$ with $i \le n$. By the nonseparation of $\Phi_{i-1}(S)$, there exists an element $(x_{i}, \ln c_{i}, \mathbf{v}) \in \Phi_{i-1}(S)$ with $v_{i} \not = 0$. As $\mathbf{e_{1}}, \ldots, \mathbf{e_{i-1}}, \mathbf{v}$ are linearly independent, applying a type $B$ automorphism gives us $(x_{i}, \ln c_{i}, \mathbf{e_{i}})$ as desired, while leaving the $(x_{j}, \ln c_{j}, \mathbf{e_{j}})$ for $j < i$ unchanged. This completed the inductive step.

We now let $\Phi'$ be the automorphism $\Phi_{n}\circ \Phi_{1} \phi$. Then, we have $(x_{i}, \ln c_{i}, \mathbf{e_{i}})$ in $\Phi'(S)$ for all $i$ from $1$ to $n$. Finally, if any $c_{i}$ are $1$, we may perturb the corresponding element by an arbitrary small amount to make the $c_{i} \not = 1$. The subset $S'$ is the preimage of the new subset under $(\Phi')^{-1}.$ Then, we compose by a type B automorphism to rescale the $\mathbf{e_{i}}$ to the desired vector $|1 - c_{i}| v_{i}$ as now each $c_{i} \not = 1$. This gives the desired automorphism $\Phi$.
\end{proof}

\begin{lemma}
    \label{kronecker}
    If $S \subseteq \R^{n}$ contains the standard basis $\mathbf{e_{i}}$, an
    element $\mathbf{v}$ with $v_{i} < 0$ and if the set $\{1, v_{1}, \ldots, v_{n}\}$
    is $\Z$-linearly independent, then $S$ is multiplicatively dense.
\end{lemma}
\begin{proof}
    Note that the convex hull of $S$ contains $0$ in the interior. Thus, $S$ is
    multiplicatively quasi-dense by \cite[Lemma 2.12]{MNT1}. Thus, the closure of the semigroup
    generated by $S$ is the closed group generated by $S$ and hence it suffices
    to prove that the group generated by $S$ is dense. Now, as $S$ contains the
    standard basis, $\overline{U}$ contains $\Z^{n}$. Hence, it suffices to
    show that $\pi(\overline{U})$ is dense, where $\pi$ is the projection onto
    $\T^{n}$. As $\{1, v_{1}, \ldots, v_{n}\}$ are $\Z$ linearly independent,
    there is no nontrivial relation $c_{0} + c_{1}v_{1} + \cdots +
    c_{n}v_{n} = 0,$ with $C_i$ integers. Thus, by a result of Kronecker cited in
    \cite{K}, $\Z\mathbf{v} \mod 1$ is dense in $\T^{n}$.
\end{proof}

\begin{lemma}
    \label{Kronecker measure 0}
    Fix nonzero $\alpha \in \R$. The set 
    $$X := \{\mathbf{v} = (v_{1}, \ldots,
    v_{n}) \in \R^{n}: \text{ $\{\alpha, v_{1}, \ldots, v_{n}\}$ is $\Z$
linearly dependent}\}$$ has measure zero in $\R^{n}.$
\end{lemma}
\begin{proof}
    Let $\phi: \R^{n} \rightarrow \R, \phi(x_{1},
    \ldots, x_{n}) = c_{0} + c_{1}x_{1} + \cdots + c_{n}x_{n}, c_{i} \in \Z,i > 0,$ and $c_{0} \in \Z\alpha$ be
    a nonzero function. If $\phi$ is constant, then
    $\phi^{-1}(0)$ is empty and is hence measure zero. If $\phi$ is
    nonconstant,  $\phi$ is a submersion as $d\phi$ is a nonzero constant.
    Thus, $\phi^{-1}(0)$ is a submanifold of codimension 1, and hence has
    measure zero. Now, $X$ is the union of $\phi^{-1}(0)$ over all
    such $\phi$. As there are only countable many such $\phi$, $X$ has
    measure zero.
\end{proof}

\begin{lemma}
    \label{euclidean greatness}
    For any $\ell$, the set of multiplicatively dense $\ell$-tuples of $\R^{n}$ is dense in the set of
    nonseparated (i.e multiplicatively quasi-dense) $\ell$-tuples of $\R^{n}$.
\end{lemma}
\begin{proof}
    As nonseparation is equivalent to multiplicatively quasi-dense in $\R^{n}$, we will prove denseness in
    multiplicatively quasi-dense $\ell$-tuples. Fix $\ell$. If there are no multiplicatively quasi-dense subsets of size $\ell$, the
    lemma is vacuously true. Suppose $S$ is a subset corresponding to a multiplicatively quasi-dense
    $\ell$-tuple. We need to find a multiplicatively dense $S'$ arbitrarily close to $S$.
    After an automorphism, $S$ contains the standard basis. Additionally, as it
    is multiplicatively quasi-dense, and $U=U(S)$ contains an element $\mathbf{w}$ with all $w_{i} > 0$,
    $\overline{U}$ and hence $U$ contains an element $\mathbf{v}$ with all $v_{i} <
    0$. Suppose $\mathbf{v} = \sum_{i=1}^{k} \mathbf{b_{i}}$ with
    $\mathbf{b_{i}} \in S$. Now, we may assume $\mathbf{b_{1}}$ is not a
    standard basis vector and that $\mathbf{b_{1}}, \ldots, \mathbf{b_{j}}$ are
    equal.

    Pulling back a neighborhood $V$ of $\mathbf{v}$ under the map $\alpha(x) =
    jx + \sum_{i = j+1}^{k} \mathbf{b_{i}}$ gives a neighborhood of
    $\mathbf{b_{1}}$.  By Lemma \ref{Kronecker measure 0}, we can find in $V$ an element $\mathbf{v'}$ such
    that each $v'_{i} < 0$ and $\{1, v'_{1}, \ldots, v'_{n}\}$ $\Z$ are linearly
    independent. Then, if $\alpha(\mathbf{b'}) = \mathbf{v'}$, replacing the subset $S$ with $S'$ by
    changing $\mathbf{b_{1}}$ to $\mathbf{b'}$, which can be done by an
    arbitrarily small perturbation, we get a multiplicatively dense subset by Lemma
    \ref{kronecker}.
\end{proof}

We now prove Semigroup Conjecture 2 for $H_{mn}$.

\begin{theorem}
    \label{greatness}
    For any $\ell$, the set of multiplicatively dense $\ell$-tuples of $H_{mn} = \R^{m} \ltimes \R^{n} = \R^{m-1}
    \times G_{n}$ is dense in the set of nonseparated $\ell$-tuples of $H_{mn}$.
\end{theorem}
\begin{proof}
   If there are no multiplicatively quasi-dense subsets of size $l$, the theorem is vacuous. So, suppose $S$ is a multiplicatively quasi-dense subset of size $l$. We will apply to $S$ an arbitrary small perturbation to get a multiplicatively dense subset. Let $\pi_{1}$ be projection onto $\R^{m} = C \times A = \R^{m-1} \times \R$, where $A$ is the $a$ coordinate of $G_{n}$. Let $\pi$ be the projection onto $G_{n}$. Applying a type $C$ automorphism, we may assume $S$ contains $(0, a, \mathbf{0})$ for $a > 0$.  By Lemma 2.12, we may perturb $S$ by an arbitrarily small amount to a subset $S'$ that contains $(w_{i}, \ln c_{i}, |1 - c_{I}| \mathbf{e_{i}})$ where each $c_{i} \not = 1$ and either $c_{1} < 1$ or $S'$ also contains $(w', \ln a', \mathbf{0})$ with $a' < 1$. Now, $S'$ may no longer be multiplicatively quasi-dense, however, $\pi_{1}(S')$ is a small perturbation of $\pi_{1}(S)$. Since the latter is multiplicatively quasi-dense by Lemma 2.5, and since the property of being multiplicatively quasi-dense in $\R^{m}$ is open (as it is equivalent to $0$ being contained in the interior of the convex hull of the subset), we see that $\pi_{1}(S')$ is still multiplicatively quasi-dense.

Now, as $\pi_{1}(S')$ is multiplicatively quasi-dense, we may, by Lemma 2.15, perturb $S'$ slightly, leaving the $B = \R^{n}$ coordinates unchanged, to get a subset $S''$ such that $\pi_{1}(S'')$ is multiplicatively dense. As consecutive arbitrary small perturbations can still be made arbitrary small, we may assume $S = S''$ by forgetting that $S$ is multiplicatively quasi-dense. However, $S$ still contains $z = (w, \ln a, \mathbf{0})$, $\ln a > 0$, $z_{i} = (w_{i}, \ln c_{i}, |1 - c_{i}|\mathbf{e_{i}})$ with $c_{i} \not = 1$, and either $c_{1} < 1$ or an additional element $z' = (w', \ln a', \mathbf{0})$ with $a' < 1$ as signs can be preserved under arbitrary small perturbations.

    Additionally, note
    that $U=U(S)$ originally contained an element $z'' = (w'', \ln c, \mathbf{b})$
    with each $b_{i} < 0, \ln c < 0$ and $b_{i + 1} < b_{i} - 2(1 - c)$. As
    this is a finite sum of elements from $S$ and the perturbation is
    arbitrarily small, $U$ still contains such an element, even
    after perturbation. Suppose $z'' = \prod_{i=1}^{k} y_{i}$. Now at least one
    of the $y_{i}$ is not $z, z', z_{i}$.  Let $\{y_{j_{1}}, \ldots,
    y_{j_{r}}\}$ be the list of all such elements. Consider the continuous map $\alpha:
    H_{mn}^{r} \rightarrow H_{mn}$ defined by
    \[
         \alpha(x_{1}, \ldots, x_{r}) = y_{1}\cdots y_{j_{1} - 1}x_{1}y_{j_{1}
         + 1} \cdots x_{r}y_{j_{r} + 1} \cdots y_{k}.
    \]
    The preimage of a neighborhood of $z''$ is a neighborhood of  $(y_{j_{1}}, \ldots, y_{j_{r}})$. Thus, using Lemma
    \ref{Kronecker measure 0}, by an arbitrarily small perturbation of each of
    these, and hence of $S$, we may change
    $z''$ to $z'' = (w'', \ln c, \mathbf{b})$ such that $\frac{\ln c}{\ln a}
    \notin \Q$, $\{1 - c, b_{1}, \ldots, b_{n}\}$ are $\Z$-linearly independent and
    $b_{i+1} < b_{i} -2(1 - c)$.  Now, $\pi_{1}(S)$ is multiplicatively dense. So by Lemma
    $\ref{exact}$, it suffices to prove $\overline{U} \cap B$ is multiplicatively dense. We
    distinguish two cases.

    Case 1: Suppose $(w, \ln a', \mathbf{0})\in S, a'<1$. By Lemma
    \ref{H_mn products} with $z',z_{i}$ for $c_{i} > 1$, and with $z,z_{i}$ for $c_{i} <
    1$, and then again by Lemma \ref{H_mn products} with $z'', z$, one has $(0, 0, \mathbf{e_{i}}),\left(0, 0,
    \frac{\mathbf{b}}{1 - c}\right)\in\overline{U} \cap B$. As $\left\{1,
    \frac{b_{1}}{1 - c}, \ldots, \frac{b_{n}}{1 - c}\right\}$ is $\Z$-linearly
    independent, by Lemma \ref{kronecker}, $\overline{U} \cap B$ is multiplicatively quasi-dense.

    Case 2: Suppose $S$ contains the element $(w_{1},\ln  c_{1}, |1 - c_{1}|\mathbf{e_
    {1}})$ with $c_{1} < 1$. Then, using Lemma \ref{H_mn products}
    $\overline{U} \cap B$ contains $(0, 0, \mathbf{e_{1}}), (0, 0,
    \frac{\mathbf{b}}{1 - c})$ and either $(0, 0, \mathbf{e_{i}})$ or $(0, 0,
    \mathbf{e_{i}} +\mathbf{e_{1}})$ for each $i$ bigger than $1.$ We now apply
    an automorphism of type C which takes $(0, 0, \mathbf{e_{i}} + \mathbf{e_{1}})$ to
    $(0, 0, \mathbf{e_{i}})$ where necessary to get the standard basis of
    $\overline{U} \cap B$. The matrix of this transformation is the identity
    with some extra -1 entries in the first row. Hence, as $b_{i+1} <
    b_{i} - 2(1-c)$, the transformation still gives an element $(0, 0,
    \mathbf{b})$ with each $b_{i} < 0$. Additionally, $\{1, b_{1}, \ldots,
    b_{n}\}$ is still $\Z$-linearly independent, as the transformation merely
    sends some $b_{i}$ to $b_{i} - b_{1}.$ Thus, applying Lemma \ref{kronecker}
    it follows that $\overline{U} \cap B$ is multiplicatively dense.
\end{proof}

\section{\bf Some examples}\label{s:examples}

One may ask whether nonseparated subsemigroups that are not groups do exist in the solvable Lie groups under consideration. The following example in $G_1$, that is the lift of a dense subsemigroup of $\R$, was suggested to us by the referee. Here we use the model for $G_1$ presented 
in Section 1, so the operation on the first component of $G_1$ is given by the usual addition of real numbers:

\begin{example} Consider the following dense, and hence nonseparated, subsemigroup of $G_1$:
\[
S=\{(x,y); (x,y)\in G_1, x=-a+b\sqrt{2}, a,b\in \N-\{0\}\}.
\]
As the inverse of any $(x,y)\in S$ does not belong to $S$, $S$ is not a group.
\end{example}

For the next two examples we use the model for $G_1$ introduced before Lemma 2.8, thus the operation on the first component of $G_1$ is given by the usual multiplication of real numbers. Note that in this case nonseparation of $S\subseteq G_1$ follows if $(1,0)$ is in the interior of the (usual) convex hull of $S$.

The following example gives a nonseparated subsemigroup of $G_1$ that is not dense. The basic idea can be used to produce many other examples.

\begin{example} Let $a=(1/3,0), b=(2,2), c=(2,-2)\in G_1$ and let $S$ be the nonseparated subsemigroup of $G_1$ generated by $a,b,c$.  The identity of $G_1$ is $(1,0)$. As any element in $S$ has on the first component a number of type $2^m/3^n, m,n\in \N,$ with at least one of $m,n$ different from zero, the inverses of $a,b,c$ are not in $S$, so $S$ is not a group.
\end{example}

In the following example, the projection of the nonseparated semigroup $S\subseteq G_1$ on the first component is a group.

\begin{example} Let $a=(1/2,0), b=(2,\sqrt{3}), c=(2,-1), d=(1,0)$ and let $S$ be the nonseparated subsemigroup of $G_1$ generated by $a,b,c,d$. The projection of $S$ on the first component is isomorphic to the infinite cyclic group $\Z$. The inverse of $(2,\sqrt{3})$ is $(1/2,-\sqrt{3}/2)$. As the coefficient of $\sqrt{3}$ in the second component of an element of $S$ is always positive, the inverse of $(2,\sqrt{3})$ is not in $S$, thus $S$ is not a group.
\end{example}

\section{\bf Minimal generating sets}\label{s:generators}

In this section we find the minimal number of generators as a closed group and as a closed semigroup of $G_n$ and of $H_{mn}$. The proofs of Proposition~\ref{p:prop1} and Proposition~\ref{p:prop2} can be used to construct more examples of dense subsemigroups.

\begin{proposition}\label{p:prop1} The smallest size of a multiplicatively quasi-dense subset (or a multiplicatively dense subset) of $G_{n}$ is $n + 2$.
\end{proposition}

\begin{proof} The smallest size of a multiplicatively quasi-dense subset and of a multiplicatively dense subset are equal due to the Semigroup Conjecture 2. Let $B$ be the $\R^n$-component of $G_n$. Suppose by contradiction that a multiplicatively quasi-dense subset $S$ has size $n + 1$. Then, as in the proof of Lemma 2.12, there exists an automorphism $\Phi$ such that $\Phi(S)$ consists of $(\ln a, \mathbf{0}),  (\ln c_{i}, \mathbf{e_{i}})$ where $\ln a > 0$ and $c_{i}>0$. This implies that $\Phi(S)$ is contained in the maximal subsemigroup defined by $B_{1} \ge 0$ where $B_{1}$ is the first component $B$. Hence, $\Phi(S)$ is separated and hence not multiplicatively quasi-dense, a contradiction.

To finish the proof, we exhibit a multiplicatively quasi-dense subset $S$ of size $n + 2$. Let $S$ consist of $(1, \mathbf{0}), (0, \mathbf{e_{i}}), 1\le i\le n,$ and  $(-\sqrt{2}, -1, \ldots, -1)$. The projection of $S$ onto the $A$ coordinate is multiplicatively quasi-dense as it is nonseparated. Thus, by Lemma 2.5, it is enough to show that $\overline{U(S)} \cap B$ is multiplicatively quasi-dense in $B$. Lemma 2.7 applied to $(1, 0), (-\sqrt{2}, -1, \ldots, -1)$ gives $z = (0, \frac{-1}{1 - e^{-1}}, \ldots, \frac{-1}{1 - e^{-1}}) \in \overline{U(S)} \cap B$. Then, as the convex hull of the elements $z,(0, \mathbf{e_{i}}),1\le i\le n,$ has $0\in B$ in the interior, we see that $\overline{U(S)} \cap B$ is multiplicatively quasi-dense in $B$, and hence $S$ is multiplicatively quasi-dense.
\end{proof}

We now prove an analog of Proposition~\ref{p:prop1} for $H_{mn}$.

\begin{proposition}\label{p:prop2} The smallest size of a multiplicatively quasi-dense subset (or a multiplicatively dense subset) of $H_{mn}$ is $\max \{n + 2, m+1\}$.
\end{proposition}

\begin{proof} Note that $H_{mn} \cong \R^{m} \ltimes \R^{n} \cong \R^{m-1} \times G_{n}$. Let $\pi, \pi_{1}$ be the projections onto $\R^{m}, G_{n}$ respectively. Let $S$ be a multiplicatively quasi-dense subset. Then, by Lemma 2.5, $\pi(S)$ and $\pi_{1}(S)$ are multiplicatively quasi-dense. Hence, $\pi(S)$ must be of size $m + 1$ or greater (as any $m$ vectors can be mapped to the standard basis by an automorphism and hence cannot form a multiplicatively quasi-dense subset). Similarly, by Proposition~\ref{p:prop1}, $\pi_{1}(S)$ must be of size $n+2$ or greater. Hence, $S$ must be of size at least $\max \{m + 1, n + 2\}$.

To finish the proof, we exhibit a multiplicatively quasi-dense subset of size $\max\{m + 1, n + 2\}.$ We distinguish two cases:

Case 1: $m + 1 \le n + 2$. $S$ consists of the following elements:
\begin{enumerate}
\item
$z_{i} = (\mathbf{e_{i}}, \ln a_{i}, \mathbf{e_{i}}), a_{i} < 1, 1\le i\le m-1$. We use $m - 1 \le n$.
\item
$z = (\mathbf{0}, \ln a, \mathbf{0})$ for some $a > 1$.
\item
$z'_{i} =  (\mathbf{0}, \ln a_{i}, |1 - a_{i}|\mathbf{e_{i}}), a_{i} < 1, m\le i\le n$.
\item
The vectors $\pi(z_{i}) = (\mathbf{e_{i}}, \ln a_{i}),1\le i\le m-1,$ and $\pi(z) = (\mathbf{0}, \ln a)$ form a basis in $\R^{m}$. Consider a vector $\mathbf{v}\in\R^{n}$ that has all coordinated negative in the above basis. Note that adding $\mathbf{v}$ to the basis gives a nonseparated set in $\R^{n}$.
We then add to $S$ an element $z' = (\mathbf{v}, |1 - e^{v_{n}}| \mathbf{w})$, where $w = (-1, \ldots, -1), b < 1$. In addition, we can choose $\mathbf{v}$ such that $v_{n} < 0$ because one of the elements in the above basis is $(\mathbf{0}, \ln a)$ for $\ln a > 0$ and we can add enough negative multiples of this vector to $\mathbf{v}$ to make $v_{n} < 0$, without affecting the other components.
\end{enumerate}

Clearly $S$ has size $n + 2=\max \{m + 1, n + 2\}$. We show that $S$ is multiplicatively quasi-dense. Note that $\pi(S)$ is nonseparated by construction. Hence, if $B$ refers to the $\R^{n}$ components, we need to show that $\overline{U(S)} \cap B$ is multiplicatively quasi-dense in $B$. As $\pi(S)$ is multiplicatively quasi-dense, we apply Lemma 2.7 to $z_{i}, z$ and $z', z$ (which by the construction and the choice of $v_{n}$ have $A$ coordinates on opposite sides of $0$) to get $\{(0, 0, \mathbf{e_{i}}), (0, 0, \mathbf{w})\}\subseteq \overline{U(S)} \cap B$.
This subset is multiplicatively quasi-dense in $B$ as it contains $0$ in the interior of its convex hull. Hence $S$ is multiplicatively quasi-dense. This finishes Case 1.

Case 2: $m + 1 \ge n + 2$. $S$ consists of the following elements:
\begin{enumerate}
\item
$z_{i} = (\mathbf{e_{i}}, \ln a_{i}, |1 - a_{i}|\mathbf{e_{i}}), \ln a_{i} < 0, 1\le i\le n$. We use $n \le m - 1$.
\item
$z = (\mathbf{0}, \ln a, \mathbf{0}), \ln a > 0$.
\item
For $i$ from $n$ to $m-1$,
$z_{i} = (\mathbf{e_{i}}, 0, \mathbf{0}).$
\item
The vectors $\pi(z_{i}), \pi(z)$ form a basis for $\R^{m}$. Let $\mathbf{v}$ be a vector with negative coordinates in this basis and set $z' = (\mathbf{v}, |1 - e^{v_{n}}| \mathbf{w})$. As in the first case, we can choose $\mathbf{v}$ such that $v_{n} < 0$.

\end{enumerate}

Clearly $S$ has size $m + 1=\max \{m + 1, n + 2\}$. As $\pi(S)$ is nonseparated by construction, it suffices to show that $\overline{U(S)} \cap B$ is multiplicatively quasi-dense in $B$, where $B$ denotes the $\R^{n}$ component of $H_{mn}$. We apply Lemma 2.7 to $z,z_{i}$ for $1\le i\le n$, and $z, z',$ to get $\{(0, 0, \mathbf{e_{i}}), (0, 0, \mathbf{w})\} \subseteq \overline{U(S)} \cap B$.
This subset is multiplicatively quasi-dense in $B$ as it contains $0$ in the interior of its convex hull. Thus, $S$ is multiplicatively quasi-dense. This finishes Case 2.
\end{proof}

We now analyze the minimal number of generators of $H_{mn}$ as a closed group. Note that $\R^{m}$ requires at least $m + 1$ generators. Indeed, any generating subset must contain $m$ linearly independent elements, but if its size is $m$, we can apply an automorphism to turn the subset into the standard basis, which generates a closed, discrete lattice $\Z^{m}$ in $\R^{m}$.

We first analyze the case of $G_{n}$.

\begin{proposition}\label{p:prop3} The minimal number of generators of $G_{n}$ as a closed group is $n + 1$.
\end{proposition}

\begin{proof} Let $p$ be the number of elements of a minimal generating set $S$. If $p\le n$, we can apply a type $A$ automorphism to make one of the elements in $S$ equal to $(\ln a, 0)$ for some $\ln a$. Then the $\R^{n}$ components of the group $\langle S \rangle$ generated by $S$ are contained in a codimension 1 hyperplane. Therefore $S$ does not generate $G_{n}$ as a group. To finish the proof, it suffices to construct a subset of size $n+1$ which generates  $G_{n}$ as a closed group. We consider the subset $S$ consisting of $z = (-1, \mathbf{0})$ and $z_{i} = (\sqrt{2}, (e^{\sqrt{2}} - 1)\mathbf{e_{i}}), 1\le i\le n$. Let $G$ be the closed subgroup of $G_{n}$ generated by $S$. Then, $G = \overline{U(S \cup S^{-1})}$. We let $B_{i}=\R \mathbf{e_{i}}$.

Now, if $\pi$ is the projection onto the $A$ coordinate, $\pi(S \cup S^{-1})$ contains $\sqrt{2}, -1$ and is hence nonseparated. Hence, we may apply Lemma 2.7, to $z, z_{i}$ to get $z'_{i} = (0, \mathbf{e_{i}}) \in \overline{U}(S \cup S^{-1}) = G$. Now, $z''_{i} = zz'_{i} = (-1,e^{-1} \mathbf{e_{i}})\in G.$
Hence, applying Lemma 2.7 to $z^{-1}, z''_{i}$ gives us $(0, \frac{e^{-1}}{1 - e^{-1}} \mathbf{e_{i}}) \in G$.

Now, as $G$ contains $(0, \frac{e^{-1}}{1 - e^{-1}} \mathbf{e_{i}})$ and $-z'_{i}$, and $\frac{e^{-1}}{1 - e^{-1}}$ is irrational, $G \cap B_{i}$ is dense and closed in $B_{i} \cong \R$ and hence coincides with $B_{i}$. Thus, $G$ contains $B_{i}, 1\le i\le n$ and hence contains the normal $\R^{n}$ subgroup of $G_{n}$, which we denote as $B$. Due to Lemma 2.5, this implies that $\pi(G)$ is closed. But $\pi(G)$ is a closed subgroup of $\R$ containing $1$ and $-\sqrt{2}$, so $\pi(G)=\R$. As $G$ contains $B$ and a representative element for each coset in $G_{n}/B$, one has $G = G_{n}$ and $p = n+1$.
\end{proof}

We prove now an analog of Proposition~\ref{p:prop3} for $H_{mn}$.

\begin{proposition} The minimal number of generators of $H_{mn}$ as a closed group is $\max\{m + 1, n + 1\}$.
\end{proposition}

\begin{proof} Let $\pi$ and $\pi_{1}$ be projections onto $\R^{m}$ and $G_{n}$ respectively. If $S$ generates $H_{mn}$ as a closed group, then $\pi(S), \pi_{1}(S)$ also generate $\R^{m}$ and $G_{n}$ as closed groups respectively. Thus, $\pi(S)$ must have size $\ge m + 1$ by the remark preceding Proposition~\ref{p:prop3}, and $\pi_{1}(S)$ must have size $\ge n+1$ by Proposition~\ref{p:prop3}. Thus, so must $S$. Hence, to prove the proposition, it suffices to construct a subset of size $\max \{m + 1, n + 1\}$ which generates $H_{mn}$ as a closed group. We distinguish two cases:

Case 1: $n + 1 \ge m + 1$. $S$ consists of the following elements:
\begin{enumerate}
\item
$z = (\mathbf{0}, 1, \mathbf{0})$.
\item
$z_{i} = (\mathbf{e_{i}}, -\sqrt{2}, |1 - e^{-\sqrt{2}}|\mathbf{e_{i}}), 1\le i\le m-1$. We use $m - 1 \le n-1$.
\item
$z_{i} = (\mathbf{0}, -\sqrt{2}, |1 - e^{-\sqrt{2}}|\mathbf{e_{i}}), m-1\le i\le n-1$.
\item
$z_{n} = (\mathbf{v}, |1 - e^{-v_{n}}|\mathbf{e_{n}})$, where $\mathbf{v}$ is a vector in $\R^{m}$ with negative and irrational coordinates relative to the basis defined by $\pi(z_{i}), \pi(z)$ for $1\le i\le m-1$. As in the previous proposition, we can choose $\mathbf{v}$ such that $v_{n} < 0$.
\end{enumerate}

This subset has size $n + 1=\max\{m + 1, n + 1\}$. Additionally, by construction, $\pi(S)$ is multiplicatively dense and hence nonseparated. Let $G = \overline{U(S \cup S^{-1})}$ be the closed group generated by $S$. Then, if we show that $G \cap B = B$ (where $B$ is the $\R^{n}$ portion of $H_{mn}$), then applying Lemma 2.5 gives that $\pi(G)$ is closed and hence, as it is dense in $\R^{m}$, coincides with $\R^{n}$. So $G$ contains $B$ and a representative of each coset in $H_{mn}/B$ and has to coincide with $H_{mn}$. It remains to show $G \cap B$ is multiplicatively dense.

Now, as $\pi(G)$ is nonseparated, we can apply Lemma 2.7 to $z, z_{i} $ to get $z'_{i} = (\mathbf{0}, \mathbf{0}, \mathbf{e_{i}}) \in G \cap B$.
On the other hand, $zz'_{i} = (-1, (1 - e^{-1}) \mathbf{e_{i}})$. We now apply the same argument as in Proposition~\ref{p:prop3} to see that $G \cap B = B$ as $G \cap B_{i} = B_{i}$ where $B_{i} = \R\mathbf{e_{i}}$. This proves that $S$ generates $H_{mn}$ as a closed group, finishing the proof under the conditions of Case 1.

Case 2: $m + 1 > n + 1$. In this case $m + 1 \ge n + 2$. So, by Case 2 of Proposition~\ref{p:prop2}, we have a multiplicatively quasi-dense subset of size $m+1$. By Theorem 2.16, we then have (by applying a small perturbation) a multiplicatively dense subset of size $m +1$. This subset generates $H_{mn}$ as a closed semigroup and hence also as a closed group.
\end{proof}

\bibliographystyle{plain}

\end{document}